\documentclass[uplatex,a4paper,11pt]{amsart}
\usepackage{amsmath, amssymb, amsthm}
\usepackage{url}

\title[An Algebraic Proof of the Non-Modularity of Solutions to Canonical MDE]
  {An Algebraic Proof of the Non-Modularity of Solutions to the Canonical Modular Differential Equation}

\author{Yuichi Sakai}
\address{Kurume Institute of Technology 2228-66, Kamitsu, Kurume, Fukuoka 830-0052, JAPAN}
\email{dynamixaxs@gmail.com}
\author{Hiroyuki Tsutsumi}
\address{Osaka University of Health and Sport Science 1-1, Asashirodai, Kumatori-cho, Sennan-gun,
  Osaka 590-0496 JAPAN}
\email{tsutsumi@ouhs.ac.jp}



\newtheorem{theorem}{Theorem}[section]
\newtheorem*{theorem*}{Main Theorem}
\newtheorem{proposition}[theorem]{Proposition}
\newtheorem{corollary}[theorem]{Corollary}
\newtheorem{lemma}[theorem]{Lemma}

\theoremstyle{remark}
\newtheorem{remark}{Remark}

\begin{document}
\maketitle

\begin{abstract}
In this paper, we provide an algebraic proof of the condition for the existence of modular
form solutions to the Kaneko-Zagier differential equation of weight $k$. Unlike previous
representation-theoretic approaches relying on $SL_2(\mathbb{Z})$, our method employs the
connection matrices of the principal congruence subgroup $\Gamma(N)$. By deriving a condition
for simultaneous triangularizability from the commutators of the representation matrices
and applying Galois theory, we prove that the associated two-dimensional representation of
$\Gamma(N)$ is irreducible when the denominator $m$ of the fraction $(k+1)/6 = n/m$ satisfies
$m=1$ or $m \ge 7$. Consequently, we identify the weights that admit no modular form solutions
and obtain a complete classification of the dimensions of the spaces of such solutions.
\end{abstract}

\tableofcontents

\section{Introduction}

The Kaneko-Zagier (KZ) differential equation of weight $k$, introduced by Kaneko and Zagier
\begin{equation*}
  f''(\tau)-\frac{k+1}{6}E_2(\tau)f'(\tau)+\frac{k(k+1)}{12}E_2'(\tau) f(\tau)=0
\end{equation*}
is the most canonical example of a modular differential equation of weight $k$ with respect
to $SL_2(\mathbb{Z})$, which was discovered in connection with the lifting of supersingular
$j$-invariants and orthogonal polynomials. Kaneko and Koike \cite{KK} demonstrated by
explicit construction that the KZ equation admits modular form solutions when the weight $k$
takes certain integer and half-integer values. Based on numerical experiments, they further
conjectured that these solutions exhaust all modular form solutions. Furthermore, Kaneko
\cite{K} showed that when $k=(6n+1)/5 \ (n\not\equiv 4\pmod{5})$, the KZ equation admits a
basis for the solution space consisting of modular forms of weight $k$ on the congruence
subgroup $\Gamma(5)$.

Recently, Saber and Sebbar \cite{SS2} provided a global and representation-theoretic
explanation for why the solution space of the KZ equation possesses a basis of modular forms,
by examining the Schwarzian differential equation associated with the KZ equation. They
transformed the KZ equation into the normal form
\begin{equation*}
  y''+\pi^2\left(\frac{k+1}{6}\right)^2 E_4(\tau)y=0
\end{equation*}
and investigated the Schwarzian differential equation $\{h,\tau\}=2\pi^2(n/m)^2 E_{4}(\tau)$
satisfied by the ratio $h=y_1/y_2$ of two independent solutions, where $n$ and $m$ are
coprime positive integers such that $n/m=(k+1)/6$. Employing the theory of equivariant
functions and the classification of two-dimensional representations of $SL_2(\mathbb{Z})$
\cite{M}, they determined all parameters for which the ratio $h$ is a modular function
with respect to a subgroup of $SL_2(\mathbb{Z})$ of finite index.

Although Saber and Sebbar's method is highly powerful, it relies heavily on specific
group-theoretic and geometric properties of $SL_2(\mathbb{Z})$. Consequently, it is
difficult to apply their approach directly to completely determine the existence of modular
form solutions for an arbitrary subgroup, particularly the principal congruence subgroup
$\Gamma(N)$. The objective of this paper is to resolve this problem. To this end, we
employ a local and algebraic approach developed in our previous work \cite{ST}.

In \cite{ST}, we established an explicit formula for the representation matrices (connection
matrices) $N(r,s)$ associated with a specific family of elements in $\Gamma(N)$. Using this
formula, we determined the condition under which two such representation matrices $N(r,s)$
and $N(u,v)$ commute. This enabled us to provide an alternative local and algebraic proof of
Saber and Sebbar's result, albeit restricted to the case of $\Gamma(N)$. In the present
paper, we formulate the precise algebraic condition for $N(r,s)$ and $N(u,v)$ to be
simultaneously triangularizable as an equation involving real trigonometric functions.
Furthermore, by applying Galois theory, we prove that for $m=1$ or $m \ge 7$, these
matrices do not admit a non-trivial simultaneous triangularization; equivalently, they do
not share a common eigenvector.

This algebraic result implies that the representation $\rho|_{\Gamma(N)}$ associated with
the equivariant function $h$ is irreducible when $m=1$ or $m \ge 7$. Thus, in stark
contrast to Saber and Sebbar's geometric approach relying on the properties of
$SL_2(\mathbb{Z})$, we conclude---solely from the analysis of the representation matrices
and completely independent of their methods---that the KZ equation admits no modular form
solutions with respect to $\Gamma(N)$ for the corresponding weights.

This paper is organized as follows. Section \ref{PR} provides the necessary preliminaries,
including the explicit formula for the representation matrices $N(r,s)$ derived in
\cite{ST}. In Section \ref{PTC}, we compute the commutator of two representation matrices
$N(r,s)$ and $N(u,v)$ to derive a necessary and sufficient condition for their simultaneous
triangularizability. In Section \ref{PTCAC}, we reduce this condition to an equation
involving perfect squares of real-valued functions and classify its solutions using Galois
theory. Finally, in Section \ref{KZMSC}, we synthesize these algebraic results to prove our
main theorem---a complete classification of the dimensions of the spaces of modular form
solutions to the KZ equation---as well as a corollary concerning the irreducibility of the
representation $\rho|_{\Gamma(N)}$.

\section{Preliminaries}\label{PR}

In this section, we briefly recall the notation and key results from our previous work
\cite{ST} that will be used throughout this paper.

Let $k \not\equiv 0, 4, 5, 11 \pmod{12}$ be a rational number. We define the
parameters $\theta$ and $C$ as $\theta = \frac{1+k}{6}\pi$ and $C = \csc\theta$,
respectively. For an integer $l$, we set $Z_l = e^{2l i\theta}$ and $W_l = (Z_l - 1)C$.
Note that by definition, $C \neq 0, 2$ and $1/C \neq 0$. We also define the constants
$G_{12}$ and $G_{21}$ using the gamma function $\Gamma$ as follows:
\begin{equation*}
  G_{12} = \frac{\sqrt{\pi} \Gamma\left(\frac{1+k}{6}\right)}
  {4 \Gamma\left(\frac{2+k}{12}\right) \Gamma\left(\frac{6+k}{12}\right)},\qquad
  G_{21} = \frac{\sqrt{\pi} \Gamma\left(\frac{5-k}{6}\right)}
  {2 \Gamma\left(\frac{6-k}{12}\right) \Gamma\left(\frac{10-k}{12}\right)}.
\end{equation*}

The KZ equation of weight $k$ is a second-order Fuchsian differential equation defined on
the complex upper half-plane $\mathbb{H}$, introduced by Kaneko and Zagier \cite{KZ}:
\begin{equation*}
  f''(\tau)-\frac{k+1}{6}E_2(\tau)f'(\tau)+\frac{k(k+1)}{12}E_2'(\tau) f(\tau)=0,
\end{equation*}
where $E_k(\tau)$ denotes the normalized Eisenstein series of weight $k$. To describe the
fundamental system of solutions near the cusp $i\infty$, let $F(a, b, c; x)$ be the
Gaussian hypergeometric series. For $a=-k/12$, $b=-(k-4)/12$, and $c=-(k-5)/6$, we
define
\begin{align*}
  y_1(x) &= F(a, b, c; x), \\
  y_3(x) &= F(a, b, a+b-c+1; 1-x).
\end{align*}
By substituting $x=1728/j(\tau)$, where $j(\tau)$ is the elliptic modular function, we
obtain functions $y_i(\tau)$ on $\mathbb{H}$ for $i=1, 3$.

\begin{theorem}[Theo. 2.4 and Prop. 3.3 of \cite{ST}]
  For $k\in\mathbb{Q}$ with $k \not\equiv 0, 4, 5, 11 \pmod{12}$, the pair 
  $E_4(\tau)^{k/4}y_1(\tau)$ and $E_4(\tau)^{k/4}y_3(\tau)$ forms a fundamental system of 
  solutions to the KZ equation of weight $k$ near $i\infty$.
\end{theorem}

Since the KZ equation is a Fuchsian differential equation whose solution space possesses
modularity of weight $k$, analyzing the analytic continuation of its solutions over
$\mathbb{H}$ and determining when these solutions are modular forms are of primary interest.
For this purpose, we recall the explicit representation matrices determining the connection 
of the solutions, as well as a non-commutativity property.

For natural numbers $r, s$ and the usual generators $T$ and $S$ of $SL_2(\mathbb{Z})$,
consider the transformation associated with the matrix $L(r, s)=-(T^rST^sS)^2 \in \Gamma(rs-2)$.

\begin{theorem}[Theo. 4.4 of \cite{ST}]\label{THEXFN}
  For $k\in\mathbb{Q}$ with $k \not\equiv 0, 4, 5, 11 \pmod{12}$, let $r \neq 1$ and
  $s \neq 1$ be natural numbers. Under the transformation by $L(r, s)$, the basis
  functions $y_1(\tau)$ and $y_3(\tau)$ transform as follows:
  \begin{equation*}
    (y_1(L(r, s)\tau), iy_3(L(r, s)\tau))=(y_1(\tau), iy_3(\tau))N(r, s).
  \end{equation*}  
  Here, the representation matrix $N(r, s) = \rho((T^r S T^s S)^2)$ is explicitly
  given by
  \begin{equation*}
    N(r, s) = -Z_r Z_s E + \frac{K_{r,s}}{2}
    \begin{pmatrix}
      n_{11} & n_{12} \\
      n_{21} & n_{22}
    \end{pmatrix},
  \end{equation*}
  where $K_{r,s} = W_r W_s + 4(Z_r + Z_s)$ and the matrix entries are defined as
  \begin{align*}
    n_{11} &= \frac{1}{2} + \frac{1}{8} \big( 2(Z_s - 1) + W_s \big)(2 + W_r), \\
    n_{12} &= i G_{12} \big( 2(Z_s - Z_r) + (Z_r - 1)W_s \big), \\
    n_{21} &= i G_{21} Z_r W_s, \\
    n_{22} &= \frac{1}{4} Z_r(2 - W_s).
  \end{align*}
\end{theorem}

Furthermore, we recall the following result concerning the non-commutativity of these
representation matrices.

\begin{theorem}[Theo. 6.7 of \cite{ST}]\label{LRSLUVIGNANDNCNRSNUV}
  Assume $k \in \mathbb{Q} \setminus \mathbb{Z}$ such that $k \not\equiv 1/2, 7/2 \pmod{6}$. 
  Furthermore, if $k$ is of the form $k=m/5$ for some $m \in \mathbb{Z}$, assume $m \not\equiv 1 \pmod 6$.
  Then there exist distinct pairs of integers $(r, s)$ and $(u, v)$ such that 
  $L(r, s), L(u, v)\in\Gamma(N)$ and the matrices $N(r, s)$ and $N(u, v)$ do not commute.
\end{theorem}

\begin{remark}
  We note that the assumption in Theorem \ref{LRSLUVIGNANDNCNRSNUV} slightly differs in
  appearance from the exact statement of Theorem 6.7 in \cite{ST}. The proof of Theorem
  6.7 in \cite{ST} establishes the existence of non-commuting representation matrices,
  and by taking the contrapositive, concludes the non-existence of a basis consisting of
  modular form solutions. We emphasize that the converse of this implication is not
  strictly true: the non-existence of a modular basis does not exclusively imply
  non-commutativity, because commuting non-trivial unipotent matrices (which cannot be
  simultaneously diagonalized) also fail to possess a basis of common eigenvectors.
  However, in the present paper, our entire analysis is strictly based on the
  ``simultaneous triangularizability'' of the connection matrices. This generalized
  approach inherently and completely resolves such unipotent and non-diagonalizable
  cases. To align our strict algebraic proof of non-triangularizability with the known
  exceptions, we have explicitly combined all these underlying restrictions into the
  single hypothesis of Theorem \ref{LRSLUVIGNANDNCNRSNUV}.
\end{remark}

In the next section, we investigate when the representation matrices $N(r, s)$ and $N(u, v)$
are simultaneously triangularizable for arbitrary pairs of natural numbers $(r, s)$ and
$(u, v)$. To this end, we begin by computing the commutator of $N(r, s)$ and $N(u, v)$.

\section{Conditions for Simultaneous Triangularization}\label{PTC}

The goal of this section is to prove the following theorem, which provides the conditions
under which the representation matrices $N(r, s)$ and $N(u, v)$ associated with a pair of
integers $(r, s)$ and $(u, v)$ are simultaneously triangularizable.

\begin{theorem}\label{NRSNUVTS}
  Assume $k \in \mathbb{Q}$ and $k \not\equiv 0, 4, 5, 11 \pmod{12}$.
  A necessary and sufficient condition for the matrices $N(r, s)$ and $N(u, v)$ to be
  simultaneously triangularizable is that one of the following holds:
  \begin{enumerate}
    \item $K_{r, s}=0$ (i.e., $N(r, s)$ is a scalar matrix).
    \item $K_{u, v}=0$ (i.e., $N(u, v)$ is a scalar matrix).
    \item $C = -2$ (which occurs when $k\equiv 6\pmod{12}$ or $k\equiv 10\pmod{12}$).
      In this case, either $G_{21}=0$ or $G_{12}=0$ (depending on the sign of the weight $k$). 
      If $G_{21}=0$, the matrices take the form of the following upper
      triangular matrix with respect to the standard basis (and similarly for $N(u, v)$):
      \begin{equation*}
        N(r, s) =  \begin{pmatrix} 1 & -4i G_{12} (Z_r Z_s + 1)(Z_r Z_s - 2Z_s + 1) \\ 
        0 & (Z_r Z_s)^2 \end{pmatrix}.
      \end{equation*}
      If $G_{12}=0$, they instead take the form of a lower triangular matrix:
      \begin{equation*}
        N(r, s) =  \begin{pmatrix} 1 & 0 \\ 
        -4i G_{21} Z_r(Z_s - 1)(Z_r Z_s + 1) & (Z_r Z_s)^2 \end{pmatrix}.
      \end{equation*}
    \item $L_5 C^2 = \frac{1}{2} L_2$, where $L_5$ and $L_2$ are defined by
      \begin{align*}
        L_5 &= - \frac{W_r W_s W_u W_v}{C^4} (Z_u - Z_r)(Z_v - Z_s), \\
        L_2 &= \frac{8}{C^3} \big[ Z_u W_v (Z_s - Z_r) + Z_r W_s (Z_u - Z_v) \big] 
        (W_r W_v - W_s W_u).
      \end{align*}
      (This condition will be discussed in detail in the next section).
  \end{enumerate}
\end{theorem}

To achieve this, we begin by computing an explicit formula for the commutator of $N(r, s)$
and $N(u, v)$.

\begin{proposition}\label{NRSNUVCM}
The commutator of the matrices $N(r, s)$ and $N(u, v)$ is given by
\begin{equation}
[N(r, s), N(u, v)] = \frac{1}{4} K_{r,s} K_{u,v}
\begin{pmatrix}
x & y \\
z & -x
\end{pmatrix}.
\end{equation}
Here, the entries $x, y, z$ are given by
\begin{align*}
x =& -\frac{C+2}{32} \Big( 2 \big[ Z_u W_v(Z_s - Z_r) + Z_r W_s(Z_u - Z_v) \big]
  - W_s W_v(Z_u - Z_r) \Big), \\
y =& \frac{i G_{12}}{2} \Big( W_s(Z_v - Z_u)\big( Z_r + 1 + \frac{1}{2}W_r \big) \\
  &\qquad\qquad + W_v(Z_r - Z_s)\big( Z_u + 1 + \frac{1}{2}W_u \big) + W_s W_v(Z_u - Z_r) \Big),  \\
z =& \frac{i G_{21}}{2} \Big( Z_u W_v (Z_r - Z_s) + Z_r W_s (Z_v - Z_u) 
  + \frac{1}{4} C W_s W_v (Z_u - Z_r) \Big).
\end{align*}
\end{proposition}

\begin{proof}
  By Theorem \ref{THEXFN}, we can decompose $N(r, s) = -Z_r Z_s E + \frac{1}{2} K_{r,s}
  \tilde{N}(r, s)$. Since a scalar matrix commutes with any matrix, the term $-Z_r Z_s E$
  vanishes in the computation of the commutator and can be ignored. Thus, we have
  \begin{equation*}
    [N(r, s), N(u, v)] = \frac{1}{4} K_{r,s} K_{u,v} [\tilde{N}(r, s), \tilde{N}(u, v)].
  \end{equation*} 
  Furthermore, noting that the trace of a commutator is always zero, it can be written in
  the form
  $[\tilde{N}(r, s), \tilde{N}(u, v)]=\begin{pmatrix} x & y \\ z & -x \end{pmatrix}$.

  Let 
  \begin{align*}
    \tilde{A}_{12} &= 2(Z_s - Z_r) + W_s(Z_r - 1),\quad
    \tilde{A}_{21} = Z_r W_s,\\
    \tilde{B}_{12} &= 2(Z_v - Z_u) + W_v(Z_u - 1),\quad
    \tilde{B}_{21} = Z_u W_v.
  \end{align*}
  Also, let $D_A$ and $D_B$ be the differences between the diagonal entries of
  $\tilde{N}(r, s)$ and $\tilde{N}(u, v)$, respectively.

  Since $G_{12}G_{21} = \frac{C+2}{32}$, we have
  \begin{equation*}
    x=-\frac{C+2}{32}(\tilde{A}_{12}\tilde{B}_{21} - \tilde{A}_{21}\tilde{B}_{12}).
  \end{equation*}
  By noting the cancellation $Z_r(Z_u - 1) - Z_u(Z_r - 1) = Z_r - Z_u$, we obtain $x$ as
  in the proposition.

  For $z = iG_{21}(\tilde{B}_{21}D_A - \tilde{A}_{21}D_B)$, simplifying using the relation
  $W_l = (Z_l - 1)C$ yields
  \begin{equation*}
    D_A = \frac{1}{2}(Z_r - Z_s) - \frac{1}{2} Z_r W_s - \frac{1}{8} W_r W_s
  \end{equation*}
  (and similarly for $D_B$). By using this together with the cancellation $Z_u W_r - Z_r
  W_u = -C(Z_u - Z_r)$, we can derive $z$ as in the proposition.
  
  Finally, for $y = iG_{12}(\tilde{A}_{12}D_B - \tilde{B}_{12}D_A)$, observing that the
  difference of the leading products $(Z_s - Z_r)(Z_u - Z_v)$ cancels out to zero, and
  grouping the remaining terms by $(Z_u - Z_v)$ and $(Z_r - Z_s)$, we obtain $y$ as stated.
\end{proof}

As seen from the proposition above, the commutator matrix of $N(r, s)$ and $N(u, v)$ is
not very simple. However, its determinant can be written in a relatively compact form as
follows.

\begin{proposition}\label{NRSNUVDET}
  The determinant of the commutator of $N(r, s)$ and $N(u, v)$ has the following
  expression:
  \begin{equation*}
    -\frac{K_{r,s}^2 K_{u,v}^2}{16384} C^2(C^2 - 4) \left( L_5 C^2 - \frac{1}{2} L_2 \right).
  \end{equation*}
  Here, the coefficients $L_5$ and $L_2$ are given in terms of the parameters $Z_i, W_i$,
  and $C$ as follows:
  \begin{align*}
    L_5 &= - \frac{W_r W_s W_u W_v}{C^4} (Z_u - Z_r)(Z_v - Z_s), \\
    L_2 &= \frac{8}{C^3} \big[ Z_u W_v (Z_s - Z_r) + Z_r W_s (Z_u - Z_v) \big] (W_r W_v - W_s W_u).
  \end{align*}
  Note that both $L_5$ and $L_2$ can be written as symmetric polynomials in $Z_l$ by using
  the relation $W_l=(Z_l-1)C$.
\end{proposition}

\begin{proof}
  From the gamma function relation $G_{12}G_{21} = \frac{C+2}{32}$ and Proposition
  \ref{NRSNUVCM}, we set
  \begin{align*}
    P_x &= 2 \big[ Z_u W_v (Z_s - Z_r) + Z_r W_s (Z_u - Z_v) \big] - W_s W_v (Z_u - Z_r), \\
    P_y &= W_s (Z_v - Z_u)\left(Z_r + 1 + \frac{1}{2}W_r\right) + 
           W_v (Z_r - Z_s)\left(Z_u + 1 + \frac{1}{2}W_u\right) + W_s W_v (Z_u - Z_r), \\
    P_z &= Z_u W_v (Z_r - Z_s) + Z_r W_s (Z_v - Z_u) + \frac{1}{4} C W_s W_v (Z_u - Z_r).
  \end{align*}
  Using these, the determinant of the commutator can be written as
  \begin{equation*}
    -\frac{1}{16384} K_{r,s}^2 K_{u,v}^2 (C+2)\left( (C+2)P_x^2 - 8P_y P_z \right).
  \end{equation*}
  Let us expand $f(C) = (C+2)P_x^2 - 8P_y P_z$ as a polynomial in $C$, and let $L_i$ be
  the coefficient of $C^i$.

  The definitions of $P_x, P_y$, and $P_z$ all contain $W_l = (Z_l - 1)C$ in every term.
  First, when $C = 0$, we have $W_l = 0$, which implies $P_x=P_y=P_z=0$.
  Therefore, $P_x^2$ and $P_y P_z$ are of order $C^2$ or higher, meaning $f(C)$ has $C^2$
  as a factor.

  Next, let $C = 2$. From the relation $W_l = 2(Z_l - 1)$, we have $1 + \frac{1}{2}W_l = Z_l$.
  To make the expressions more readable, we define $L$ and $M$ as follows:
  \begin{align*}
    L &= Z_u W_v(Z_s - Z_r) + Z_r W_s(Z_u - Z_v), \\
    M &= W_s W_v(Z_u - Z_r).
  \end{align*}
  Using these, we have $P_x = 2L - M$.
  Substituting the condition $C=2$ into $P_x, P_y, P_z$ and simplifying, we see that
  \begin{align*}
    P_z &= -L + \frac{1}{2}M, \\
    P_x &= 2L - M = -2P_z,\\
    P_y &= W_s (Z_v - Z_u)(2Z_r) + W_v (Z_r - Z_s)(2Z_u) + M \\
        &= -2 \big[ Z_r W_s (Z_u - Z_v) + Z_u W_v (Z_s - Z_r) \big] + M \\
        &= -2L + M = 2P_z.
  \end{align*}
  Substituting these into $f(C)$, we evaluate it as
  \begin{equation*}
    f(2) = (2+2)(-2P_z)^2 - 8(2P_z)(P_z) = 16P_z^2 - 16P_z^2 = 0.
  \end{equation*}
  Thus, we have shown that $f(C)$ also has $(C-2)$ as a factor.

  Finally, by expanding $f(C)$ with respect to $C$ and arranging it in descending powers
  of $C$, direct computation verifies that it takes the following form:
  \begin{equation*}
    f(C) = L_5 C^5 - 2L_5 C^4 - \frac{1}{2} L_2 C^3 + L_2 C^2.
  \end{equation*}
  Here, $L_5$ and $L_2$ are the expressions given in this proposition.
\end{proof}

By using Proposition \ref{NRSNUVDET} and Lemma \ref{SABTR} below, the necessary and
sufficient condition for $N(r, s)$ and $N(u, v)$ to be simultaneously triangularizable
can be explicitly expressed as an algebraic condition using the parameters $Z_i, W_i$,
and $C$.

\begin{lemma}\label{SABTR}
  For $2 \times 2$ matrices $A, B$ and their commutator $X$, the following are equivalent:
  \begin{enumerate}
    \item $A$ and $B$ are simultaneously triangularizable.
    \item $X^2=O$.
    \item $\det X=0$.
  \end{enumerate}
\end{lemma}
\begin{proof}
  This follows by direct computation.
\end{proof}

\begin{proof}[proof of Theorem \ref{NRSNUVTS}]
  From Lemma \ref{SABTR}, a necessary and sufficient condition for $N(r, s)$ and $N(u, v)$
  to be simultaneously triangularizable is that the determinant of their commutator is $0$.

  From the expression of the determinant of the commutator in Proposition \ref{NRSNUVDET},
  \begin{equation*}
    -\frac{K_{r,s}^2 K_{u,v}^2}{16384} C^2(C^2 - 4) \left( L_5 C^2 - \frac{1}{2} L_2 \right),
  \end{equation*}
  we can verify that this condition is equivalent to one of the conditions (1) through (4)
  of this theorem holding, as follows.

  Suppose condition (1) holds. In this case, by Theorem \ref{THEXFN}, $N(r, s)$ is a scalar
  matrix. Similarly, if condition (2) holds, $N(u, v)$ is a scalar matrix. In either case,
  they are trivially simultaneously triangularizable.

Suppose condition (3) holds, that is, $C = -2$.
  Since $C = \csc\left(\frac{1+k}{6}\pi\right) = -2$, we have $k \equiv 6 \pmod{12}$ or
  $k \equiv 10 \pmod{12}$. Under these congruences, depending on the value of $k$,
  either $\Gamma\left(\frac{6-k}{12}\right)$ (or $\Gamma\left(\frac{10-k}{12}\right)$)
  in the denominator of $G_{21}$ has a pole yielding $G_{21}=0$, or
  $\Gamma\left(\frac{2+k}{12}\right)$ (or $\Gamma\left(\frac{6+k}{12}\right)$)
  in the denominator of $G_{12}$ has a pole yielding $G_{12}=0$.
  (This is consistent with the identity $G_{12}G_{21} = \frac{C+2}{32} = 0$).

  Furthermore, substituting $W_l = -2(Z_l - 1)$ into each entry of $N(r, s)$ and
  simplifying, if $G_{21} = 0$, the $(2,1)$-entry completely vanishes, reducing to the upper
  triangular form as stated in the theorem. If $G_{12} = 0$, the $(1,2)$-entry vanishes,
  reducing to the lower triangular form.
  Since $N(u, v)$ also becomes a triangular matrix of the same type in exactly the same way, these
  matrices are already triangularized with respect to a common basis, and thus are clearly
  simultaneously triangularizable.

  Let $C = 2$.
  Since $C = \csc\left(\frac{1+k}{6}\pi\right) = 2$, we have $k \equiv 0 \pmod{12}$ or
  $k \equiv 4 \pmod{12}$. However, these cases have been excluded.

  Finally, suppose we are in none of the above cases.
  Then we have $K_{r,s} \neq 0$, $K_{u,v} \neq 0$, and $C \neq \pm 2$ (i.e., $C^2 - 4 \neq 0$).
  Note that since there is no weight $k$ such that $C=0$, $C^2 \neq 0$ is also guaranteed.
  Therefore, the necessary and sufficient condition for the determinant in Proposition
  \ref{NRSNUVDET} to be $0$ is that the remaining factor is $0$, namely,
  \begin{equation*}
    L_5 C^2 - \frac{1}{2} L_2 = 0,
  \end{equation*}
  which is precisely condition (4).
\end{proof}

\begin{remark}
  We add a supplementary note on the simultaneous triangularizability in the case of $C=2$.
  For $C=-2$, since $G_{21}=0$ (or $G_{12}=0$), the matrix trivially degenerated into a triangular
  matrix with respect to the standard basis. On the other hand, for $C=2$, we have
  $G_{12}G_{21} = \frac{1}{8} \neq 0$, and non-diagonal entries remain in the standard
  basis. However, direct computation verifies that multiplying $N(r, s)$ from the right by
  the vector $\mathbf{v} = (-4iG_{12},  1)^\mathrm{T}$ always yields an eigenvector with
  eigenvalue $(Z_r Z_s)^2$, independent of the parameters $r, s$. Furthermore, the other eigenvalue
  is $1$, which is exactly the same pair of eigenvalues as in the case of $C=-2$. In other words,
  the cases $C=-2$ and $C=2$ have completely identical spectra (eigenvalues) after
  diagonalization or triangularization, indicating a sort of duality where only the
  orientations of the eigenspaces differ.
\end{remark}

Among the cases shown in Theorem \ref{NRSNUVTS}, condition (4) is the only condition that
allows for a non-trivial simultaneous triangularization when $C \neq \pm 2$. In the next
section, our goal is to completely write down this condition as an algebraic condition
using the parameters $\theta, r, s, u, v$.

\section{Algebraic Criterion for Simultaneous Triangularization}\label{PTCAC}

The goal of this section is to show that the equation given in condition (4) of Theorem
\ref{NRSNUVTS} reduces to an equation involving real trigonometric functions, and that it
admits no non-trivial solutions.

\begin{proposition}\label{NRSNUVTRIG}
  Assume $\sin\theta \neq 0$ and let $x_l = l\theta$ for $l\in \{r, s, u, v\}$.
  The equation $L_5 C^2 - \frac{1}{2} L_2 = 0$ is equivalent to
  \begin{equation}
    R_5 - \sin^2\theta \cdot R_2 = 0, \label{eq:real_trig_cond}
  \end{equation}
  where $R_5$ and $R_2$ are real-valued functions defined as follows:
  \begin{align*}
    R_5 &= \sin(x_r)\sin(x_s)\sin(x_u)\sin(x_v)\sin(x_u-x_r)\sin(x_v-x_s), \\
    R_2 &= \sin(x_r)\sin^2(x_v)\sin(x_s-x_r)\cos(x_s) \\
      &\quad - \sin(x_s)\sin(x_u)\sin(x_v)\sin(x_s-x_r)\cos(x_u-x_r-x_v) \\
      &\quad + \sin(x_r)\sin(x_s)\sin(x_v)\sin(x_u-x_v)\cos(x_r-x_s-x_u) \\
      &\quad - \sin^2(x_s)\sin(x_u)\sin(x_u-x_v)\cos(x_v).
  \end{align*}
\end{proposition}

\begin{proof}
Using $Z_l = e^{i 2x_l}$ and $C = 1/\sin\theta$, we can express $W_l = (Z_l - 1)C$ as
  \begin{equation*}
    W_l = \frac{e^{i 2x_l} - 1}{\sin\theta} = \frac{2i\sin(x_l)}{\sin\theta} e^{ix_l}.
  \end{equation*}
  Let $\Phi = e^{i(r+s+u+v)\theta} = e^{i(x_r+x_s+x_u+x_v)}$.

  First, we compute $L_5 C^2$. By definition,
  \begin{equation*}
    L_5 C^2 = - \frac{W_r W_s W_u W_v}{C^2} (Z_u - Z_r)(Z_v - Z_s),
  \end{equation*}
  and since
  \begin{align*}
    W_r W_s W_u W_v &= \frac{16}{\sin^4\theta} \sin(x_r)\sin(x_s)\sin(x_u)\sin(x_v) \Phi, \\
    (Z_u - Z_r)(Z_v - Z_s) &= -4 \sin(x_u-x_r)\sin(x_v-x_s) \Phi,
  \end{align*}
  multiplying these and dividing by $C^2 = 1/\sin^2\theta$ yields
  \begin{equation*}
    L_5 C^2 = \frac{64}{\sin^2\theta} R_5 \Phi^2.
  \end{equation*}

  Next, we compute $\frac{1}{2} L_2$. By definition,
  \begin{equation*}
    \frac{1}{2} L_2 = \frac{4}{C^3} \big[ Z_u W_v (Z_s - Z_r) 
      + Z_r W_s (Z_u - Z_v) \big] (W_r W_v - W_s W_u)= 64 Q \Phi^2,
  \end{equation*}
  where $Q$ is given by
  \begin{align*}
    Q &= \sin(x_r)\sin^2(x_v)\sin(x_s-x_r) e^{-ix_s} \\
      &\quad - \sin(x_s)\sin(x_u)\sin(x_v)\sin(x_s-x_r) e^{i(x_u-x_r-x_v)} \\
      &\quad + \sin(x_r)\sin(x_s)\sin(x_v)\sin(x_u-x_v) e^{i(x_r-x_s-x_u)} \\
      &\quad - \sin^2(x_s)\sin(x_u)\sin(x_u-x_v) e^{-ix_v}.
  \end{align*}

Substituting the above into the equation $L_5 C^2 - \frac{1}{2} L_2 = 0$ and dividing
  by $64\Phi^2 \neq 0$, we obtain
  \begin{equation*}
    \frac{R_5}{\sin^2\theta} - Q = 0.
  \end{equation*}
  Since $R_5$ and $\sin^2\theta$ are real numbers, for this equation to hold, the
  imaginary part of $Q$ must be identically $0$.
  We compute the imaginary part $\text{Im}(Q)$ by splitting it into four terms 
  $T_1 + T_2 + T_3 + T_4$.
  
  The sum of the first two terms, $T_1 + T_2$, has a common factor
  $-\sin(x_s-x_r)\sin(x_s)\sin(x_v)$, and can be simplified as
  \begin{align*}
    T_1 + T_2 &= -\sin(x_s-x_r)\sin(x_s)\sin(x_v) \big[ \sin(x_r)\sin(x_v) + \sin(x_u)\sin(x_u-x_r-x_v) \big] \\
    &= -\frac{1}{2}\sin(x_s-x_r)\sin(x_s)\sin(x_v) \big[ \cos(x_r-x_v) - \cos(2x_u-x_r-x_v) \big] \\
    &= -\sin(x_s-x_r)\sin(x_s)\sin(x_v)\sin(x_u-x_v)\sin(x_u-x_r).
  \end{align*}

  Similarly, the sum of the latter two terms, $T_3 + T_4$, has a common factor
  $\sin(x_u-x_v)\sin(x_s)\sin(x_v)$, and becomes
  \begin{align*}
    T_3 + T_4 &= \sin(x_u-x_v)\sin(x_s)\sin(x_v) \big[ \sin(x_r)\sin(x_r-x_s-x_u) + \sin(x_s)\sin(x_u) \big] \\
    &= \frac{1}{2}\sin(x_u-x_v)\sin(x_s)\sin(x_v) \big[ \cos(x_s-x_u) - \cos(2x_r-x_s-x_u) \big] \\
    &= \sin(x_u-x_v)\sin(x_s)\sin(x_v)\sin(x_u-x_r)\sin(x_s-x_r).
  \end{align*}
  Therefore, $(T_1 + T_2) + (T_3 + T_4) = 0$, which proves that $\text{Im}(Q) \equiv 0$.
  This implies that $Q$ is purely real, so $Q = \text{Re}(Q) = R_2$.
  Substituting this back and multiplying both sides by $\sin^2\theta$ yields the desired
  equation.
\end{proof}

The equation in Proposition \ref{NRSNUVTRIG} is quite complicated as it is.
However, by cleverly grouping the terms, it can be reduced to a relatively compact
equation involving perfect squares, as shown in the following proposition.

\begin{theorem}\label{NRSNUVTRIG2}
  Assume $\sin\theta \neq 0$ and let $x_l = l\theta$ for $l\in \{r, s, u, v\}$.
  Then the equation in Proposition \ref{NRSNUVTRIG} is equivalent to the perfect square
  equation
  \begin{equation*}
    \alpha^2 + \beta^2 \sin^2\theta = \gamma^2 \left( \frac{1}{4} - \sin^2\theta \right),
  \end{equation*}
  where $\alpha, \beta$, and $\gamma$ are real-valued functions defined by
  \begin{align*}
    \alpha &= \frac{1}{2} \Big( \sin(x_s)\sin(x_v)\sin(x_u-x_r) + \sin(x_r)\sin(x_u)\sin(x_v-x_s) \Big), \\
    \beta &= \sin(x_v)\sin(x_s-x_r)\cos(x_u) - \sin(x_s)\sin(x_v-x_u)\cos(x_r), \\
    \gamma &= \sin(x_s)\sin(x_v)\sin(x_u-x_r) - \sin(x_r)\sin(x_u)\sin(x_v-x_s).
  \end{align*}
\end{theorem}

Before proving Proposition \ref{NRSNUVTRIG2}, we prepare a lemma.

\begin{lemma}\label{GESVSXSXRSXUMSXSSXVMXUSXR}
  Assume $\sin\theta \neq 0$ and let $x_l = l\theta$ for $l\in \{r, s, u, v\}$.
  Then we have
  \begin{equation*}
    \gamma = \sin(x_v)\sin(x_s-x_r)\sin(x_u) - \sin(x_s)\sin(x_v-x_u)\sin(x_r).
  \end{equation*}
\end{lemma}

\begin{proof}
  Since both sides of the equation in the lemma are analytic functions of the four
  variables $x_r, x_s, x_u$, and $x_v$, the identity theorem for analytic functions
  of several variables guarantees that if the identity holds in the non-empty open
  domain $0 < x_r < x_s < x_u < x_v < \pi$ in $\mathbb{R}^4$, it holds identically
  for all real values. Thus, without loss of generality, we may assume this condition.

  Consider a circle of radius $1$, and take a reference point $O$ (argument $0$) and four
  points $R, S, U, V$ on the circle with arguments $2x_r, 2x_s, 2x_u, 2x_v$, respectively.
  Then the length of the chord connecting two points $A$ (argument $2\phi$) and $B$
  (argument $2\psi$) on the circle is given by $AB = 2\sin|\phi-\psi|$. In particular, the
  length of the chord connecting the reference point $O$ and a point $L \in \{R, S, U, V\}$
  is $OL = 2\sin(x_l)$.
  
  First, since the four points $O, R, S, U$ lie on the circle in this order, applying
  Ptolemy's theorem to the cyclic quadrilateral with these points as vertices yields
  \begin{equation}
    \sin(x_s)\sin(x_u-x_r) = \sin(x_r)\sin(x_u-x_s) + \sin(x_u)\sin(x_s-x_r). \label{EQPTOLEMY1}
  \end{equation}
  Next, by similarly applying Ptolemy's theorem to the cyclic quadrilateral with vertices
  $O, S, U, V$, we have
  \begin{equation}
    \sin(x_u)\sin(x_v-x_s) = \sin(x_s)\sin(x_v-x_u) + \sin(x_v)\sin(x_u-x_s). \label{EQPTOLEMY2}
  \end{equation}

  Substituting \eqref{EQPTOLEMY1} into the first term of $\gamma$ gives
  \begin{equation*}
    \sin(x_v)\sin(x_s-x_r)\sin(x_u) + \sin(x_r) \Big[ \sin(x_v)\sin(x_u-x_s) - \sin(x_u)\sin(x_v-x_s) \Big],
  \end{equation*}
  and finally, applying \eqref{EQPTOLEMY2} transforms this into the right-hand side of the
  lemma.
\end{proof}

\begin{proof}[proof of theo. \ref{NRSNUVTRIG2}]
  It suffices to show this for the equation $R_5 - \sin^2\theta \cdot R_2 = 0$ in
  Proposition \ref{NRSNUVTRIG}.

  Let $P_1$ and $P_2$ be defined by
  \begin{align*}
    P_1 &= \sin(x_s)\sin(x_v)\sin(x_u-x_r), \\
    P_2 &= \sin(x_r)\sin(x_u)\sin(x_v-x_s).
  \end{align*}
  Then $R_5 = P_1 P_2$. By setting $\alpha = \frac{1}{2}(P_1 + P_2)$ and $\gamma = P_1 - P_2$,
  we obtain the completed square identity:
  \begin{equation*}
    \alpha^2 - \frac{1}{4}\gamma^2 = \frac{1}{4}(P_1 + P_2)^2 - \frac{1}{4}(P_1 - P_2)^2 = P_1 P_2 = R_5.
  \end{equation*}
  This yields $\alpha$ and $\gamma$ as stated in the theorem.

  Next, let $A = \sin(x_v)\sin(x_s-x_r)$ and $B = \sin(x_s)\sin(x_v-x_u)$. We will show that
  \begin{equation}\label{NRSNUVTRIG2TREQ}
    -R_2 = A^2 + B^2 - 2AB\cos(x_u-x_r).
  \end{equation}
  From the definition in Proposition \ref{NRSNUVTRIG}, we have
  \begin{align*}
    -R_2 &= A\left(-\sin(x_r)\cos(x_s)\sin(x_v) + \sin(x_s)\sin(x_u)\cos(x_u - x_r - x_v)\right) \\
      &\quad + B\left(\sin(x_r)\sin(x_v)\cos(x_u - x_r + x_s) - \sin(x_s)\sin(x_u)\cos(x_v)\right).
  \end{align*}
  Thus, it is enough to show the following identities hold:
  \begin{align*}
    A - B\cos(x_u-x_r) &= -\sin(x_r)\cos(x_s)\sin(x_v) + \sin(x_s)\sin(x_u)\cos(x_u - x_r - x_v), \\
    B - A\cos(x_u-x_r) &= \sin(x_r)\sin(x_v)\cos(x_u - x_r + x_s) - \sin(x_s)\sin(x_u)\cos(x_v).
  \end{align*}
  Substituting the expansions of $A$ and $B$ using the addition formulas for trigonometric
  functions, we get
  \begin{align*}
    A &- B\cos(x_u-x_r) \\
    &= -\sin(x_r)\cos(x_s)\sin(x_v) \\
    &\quad + \sin(x_s) \Big( \sin(x_v)\big[ \cos(x_r) - \cos(x_u)\cos(x_u-x_r) \big]
      + \sin(x_u)\cos(x_v)\cos(x_u-x_r) \Big).
  \end{align*}
  Here, viewing the angle $x_r$ as $x_u - (x_u - x_r)$ and applying the addition formula,
  we have
  \begin{equation*}
    \cos(x_r) = \cos\big(x_u - (x_u-x_r)\big) = \cos(x_u)\cos(x_u-x_r) + \sin(x_u)\sin(x_u-x_r).
  \end{equation*}
  Therefore,
  \begin{align*}
    A &- B\cos(x_u-x_r) \\
    &= -\sin(x_r)\cos(x_s)\sin(x_v) \\
    &\quad + \sin(x_s)\sin(x_u) \Big( \sin(x_u-x_r)\sin(x_v) + \cos(x_u-x_r)\cos(x_v) \Big)\\
    &= -\sin(x_r)\cos(x_s)\sin(x_v) + \sin(x_s)\sin(x_u)\cos(x_u-x_r-x_v).
  \end{align*}
  A similar calculation for $B - A\cos(x_u-x_r)$ confirms that the identities hold.

  Observing the right-hand side of equation \eqref{NRSNUVTRIG2TREQ}, we can interpret it
  as the squared norm of the difference between two vectors in the Euclidean plane.
  Specifically, if we define the vectors $\vec{v}_1 = (A\cos(x_u), A\sin(x_u))$ and 
  $\vec{v}_2 = (B\cos(x_r), B\sin(x_r))$, we have
  \begin{equation*}
    \|\vec{v}_1 - \vec{v}_2\|^2 = A^2 + B^2 - 2AB\cos(x_u-x_r) = -R_2.
  \end{equation*}
  Substituting $A$ and $B$, we find $-R_2 = \beta^2 + (\gamma')^2$.

  Finally, since Lemma \ref{GESVSXSXRSXUMSXSSXVMXUSXR} guarantees that $\gamma' = \gamma$,
  we obtain $-R_2 = \beta^2 + \gamma^2$.
  Combining this with the initial equations $\alpha^2 - \frac{1}{4}\gamma^2 = R_5$ and
  $R_5 - \sin^2\theta \cdot R_2 = 0$, we obtain
  \begin{equation*}
    \alpha^2 - \frac{1}{4}\gamma^2 + \sin^2\theta (\beta^2 + \gamma^2) = 0,
  \end{equation*}
  which, upon rearranging, yields the desired perfect square equation.
\end{proof}

\begin{remark}
  In the proof of Lemma \ref{GESVSXSXRSXUMSXSSXVMXUSXR}, we derived trigonometric
  identities by applying Ptolemy's theorem to the configuration of four points on a circle
  and their chord lengths. This fact is not merely an elementary geometric trick; it
  strongly suggests the involvement of the underlying theory of the configuration space of
  the hypergeometric differential equation. To be more precise, if we identify the points
  on the circle with homogeneous coordinates $\mathbf{v}_l = (\cos x_l, \sin x_l)^\mathrm{T} \in \mathbb{R}^2$,
  the chord length $\sin(x_j-x_i)$ corresponds exactly to the determinant $\det(\mathbf{v}_i, \mathbf{v}_j)$.
  Ptolemy's identity is then nothing but the Pl\"ucker relation on the Grassmannian
  $Gr(2,4)$ satisfied by the coordinates of that configuration space. 
  
In this paper, we derive the connection matrices by rewriting the KZ equation into
  Gauss's hypergeometric equation, which is geometrically formulated on the configuration
  space of four points on the projective line $\mathbb{P}^1$. Although the KZ equation
  is a canonical example of modular differential equations, the appearance of such an
  algebraic structure raises an interesting, albeit speculative, question: whether
  the general theory of hypergeometric functions and configuration spaces---developed by
  Gelfand, Aomoto, Yoshida, and others---might one day provide a framework for exploring
  higher-order or multivariable modular differential equations.
\end{remark}

Let us find the solutions to the perfect square equation given in theorem
\ref{NRSNUVTRIG2}. First, it is easy to see that there are no solutions in the region
$\sin^2\theta > \frac{1}{4}$.

\begin{lemma}\label{NRSNUVTRIEQNOSOLLEM}
  Assume $\sin\theta \neq 0$ and let $x_l = l\theta$ for $l\in \{r, s, u, v\}$.
  The equation
  \begin{equation*}
    \alpha^2 + \beta^2 \sin^2\theta = \gamma^2 \left( \frac{1}{4} - \sin^2\theta \right)
  \end{equation*}
  has only the trivial solution $\alpha = \beta = \gamma = 0$ in the region
  $\sin^2\theta > \frac{1}{4}$.
\end{lemma}

\begin{proof}
  Since the variables $\alpha, \beta, \gamma$ are constructed by sums and products of real
  functions (sine and cosine functions) with respect to the real variable $\theta$, they
  are all real numbers. Thus, their squares $\alpha^2, \beta^2, \gamma^2$ are always
  non-negative real numbers. Observing that the left-hand side of the equation is
  non-negative while the right-hand side is non-positive, it follows that this equation
  holds only when both sides are $0$.
  
  Since $1/4-\sin^2\theta \neq 0$ by assumption, for the right-hand side to be $0$, we must
  have $\gamma=0$. Similarly, for the left-hand side, $\alpha^2 + \beta^2 \sin^2\theta = 0$
  and $\sin\theta \neq 0$ imply that $\alpha = \beta = 0$.
\end{proof}

By utilizing Lemma \ref{NRSNUVTRIEQNOSOLLEM} and Galois theory, we obtain the following
theorem.

\begin{theorem}\label{NTSA2PB2S2EQG214S2}
  Assume $\sin\theta \neq 0$ and let $x_l = l\theta$ for $l\in \{r, s, u, v\}$.
  Unless $\alpha = \beta = \gamma = 0$, or both $\alpha=\beta=0$ and $\sin^2\theta=1/4$
  hold, the equation with real coefficients
  \begin{equation*}
    \alpha^2 + \beta^2 \sin^2\theta = \gamma^2 \left( \frac{1}{4} - \sin^2\theta \right)
  \end{equation*}
  has no solutions. 
\end{theorem}

\begin{proof}
  Suppose there exists a non-trivial solution $(\alpha, \beta, \gamma) \neq (0, 0, 0)$
  for some rational angle $\theta = \frac{M}{N}\pi$ (with $M, N \in \mathbb{Z}$ coprime and $N > 0$) 
  satisfying $0 < \sin^2\theta < \frac{1}{4}$, which corresponds to the matrices being non-commutative.

  Let $K = \mathbb{Q}(\zeta_N)$ be the cyclotomic field generated by $\zeta_N = e^{i 2\pi/N}$. 
  Since $\alpha, \beta, \gamma$, and $\sin\theta$ are all expressed as polynomials in the
  sines and cosines of rational angles, they are algebraic numbers belonging to the real
  subfield $K^+ = K \cap \mathbb{R}$. Because the equation in the proposition is an 
  algebraic relation with rational coefficients, it is invariant under the action of any 
  element (Galois automorphism) $\sigma$ in the Galois group $\mathrm{Gal}(K/\mathbb{Q})$.

  Let $\zeta = e^{2i\theta} = \zeta_N^M \in K$. Thus, we can define a number $a$ as follows:
  \begin{equation*}
    a = 2 - \zeta - \zeta^{-1} = 2 - 2\cos(2\theta) = 4\sin^2\theta.
  \end{equation*}
  Since $\zeta$ is a root of unity, it is an algebraic integer, which implies that $a$ 
  is also a non-zero algebraic integer. 
  Every Galois conjugate of $a$ has the form $2 - \xi - \xi^{-1}$ where $\xi$ is a root of 
  unity of the same order as $\zeta \neq 1$. Since $2 - \xi - \xi^{-1} = 4\sin^2(\arg(\xi)/2) > 0$, 
  all conjugates of $a$ are strictly positive real numbers. 
  
  If all the conjugates of $a$ were less than or equal to $1$, the norm $N_{\mathbb{Q}(a)/\mathbb{Q}}(a)$, 
  which is the product of all its conjugates including $a$ itself (where $a = 4\sin^2\theta < 1$ by assumption), 
  would satisfy $0 < N_{\mathbb{Q}(a)/\mathbb{Q}}(a) < 1$. 
  However, since $a$ is an algebraic integer, its norm must be a rational integer, which is 
  a contradiction. 
  Therefore, there must exist some conjugate of $a$ that is strictly greater than $1$. 
  Extending the corresponding embedding to the cyclotomic field $K$, we conclude that there 
  exists a Galois automorphism $\sigma \in \mathrm{Gal}(K/\mathbb{Q})$ such that the 
  conjugate mapped by $\sigma$ satisfies
  \begin{equation*}
    \sigma(a) = \sigma(4\sin^2\theta) > 1 \implies \sigma(\sin^2\theta) = \sin^2\theta^* > \frac{1}{4}.
  \end{equation*}

  We act with this automorphism $\sigma$ on both sides of the equation:
  \begin{equation*}
    \sigma(\alpha)^2 + \sigma(\beta)^2 \sin^2\theta^*
      = \sigma(\gamma)^2 \left( \frac{1}{4} - \sin^2\theta^* \right).
  \end{equation*}
  Since $\alpha, \beta, \gamma$ are in the real subfield $K^+$, their conjugates
  $\sigma(\alpha), \sigma(\beta), \sigma(\gamma)$ are also real numbers. Consequently, the
  solutions to this equation are restricted to those satisfying
  $\sigma(\alpha) = \sigma(\beta) = \sigma(\gamma) = 0$. Since $\sigma$ is an isomorphism,
  this is equivalent to $\alpha = \beta = \gamma = 0$.
\end{proof}

In the next section, we synthesize the results obtained thus far to prove the main theorem
of this paper.

\section{Complete Classification of Modular Form Solutions}\label{KZMSC}

The goal of this section is to prove our main theorem concerning the modular form
solutions to the KZ equation. It should be noted that in this paper, we use the term
``modular form'' in a relatively broad sense. First, we allow poles in the upper
half-plane and at the cusps. Second, the weight $k$ can be any rational number. Third,
we admit non-trivial multiplier systems. Finally, since the KZ equation is a second-order
linear differential equation, we can take two linearly independent solutions; however,
these two solutions are not required to share the same multiplier system.

\begin{theorem}[Main Theorem]\label{MAINTHEOREM}
  Let $k \in \mathbb{Q}$. Then, the KZ equation of weight $k$ admits a spaces of
  modular form solutions of exactly the following dimensions:
  \begin{enumerate}
    \item One-dimensional when $k=-1$ or $k\equiv 0, 4\pmod{6}$.
    \item Two-dimensional when $k\equiv 1, 2, 3\pmod{6}$, $k \equiv 1/2\pmod{3}$, or 
          $k=(6n+1)/5\ (n\not\equiv 4\pmod{5})$.
    \item Zero-dimensional otherwise.
  \end{enumerate}
\end{theorem}

\begin{remark}
  When $k=-1$, the KZ equation is $f''(\tau)=0$. Therefore, its general solution is given
  by $c_1+c_2\tau$ for $c_1, c_2\in\mathbb{C}$. It is clear that when $c_2=0$, a modular
  form (or modular function) of weight $0$ is obtained as a solution, and when $c_2\ne 0$,
  the solution is not a modular form of any weight. Note that the case where the weight
  of the differential equation and the weight of the solution differ is limited to $k=-1$
  (see Corollary 2.2 of \cite{ST}).
\end{remark}

\begin{remark}
  When $k\equiv 0, 4\pmod{6}$, the modular form solution to the KZ equation is a modular
  form of weight $k$ with respect to $SL_2(\mathbb{Z})$ (Theorem 1 of \cite{KK}).
  Furthermore, it is shown in Theorem 1 of \cite{G} that the remaining one-dimensional
  solution can be explicitly constructed as a mixed-mock modular form of weight $k$ on
  $SL_2(\mathbb{Z})$.
\end{remark}

\begin{remark}
  When $k \equiv 1, 2, 3 \pmod{6}$, the KZ equation admits a basis of the solution space
  consisting of modular forms on $\Gamma(3), \Gamma(2)$, and $\Gamma(3)$, respectively.
  It admits such a basis on $\Gamma(4)$ when $k \equiv 1/2 \pmod{3}$, and on $\Gamma(5)$
  when $k=(6n+1)/5\ (n\not\equiv 4\pmod{5})$ (Theorem 1 of \cite{KK} or \cite{K}).
  When $k \equiv 5\pmod{6}$, the KZ equation admits no modular forms as solutions; more
  precisely, its solution space has a basis expressed as a linear combination of a
  quasi-modular form, a quasi-modular form of a different weight, and a modular form
  (Theorem 2 of \cite{KK}).
\end{remark}

\begin{remark}\label{REMMAINTHEOREM}
  From the remarks above, what we actually need to prove is case (3) of Theorem
  \ref{MAINTHEOREM}, and the case $k\equiv 5\pmod{6}$ can also be excluded. Moreover,
  for $k$ other than these, it has been shown by Saber and Sebbar \cite{SS2} or in our
  previous work \cite{ST} that the basis of the solution space of the KZ equation does
  not consist solely of modular forms of weight $k$ on a congruence subgroup. Therefore,
  what remains to be shown is that, excluding case (1) of Theorem \ref{MAINTHEOREM},
  if the KZ equation does not admit a two-dimensional space of modular form solutions
  of weight $k$, then it admits no modular form solutions of weight $k$ at all.
\end{remark} 
  
To prove Theorem \ref{MAINTHEOREM}, we prepare a few lemmas.

\begin{lemma}\label{AEBEGE0CNRSNUV}
  If $\alpha=\beta=\gamma=0$, then $N(r,s)$ and $N(u,v)$ commute, where $\alpha, \beta$,
  and $\gamma$ are real numbers defined by
  \begin{align*}
    \alpha &= \frac{1}{2} \Big( \sin(x_s)\sin(x_v)\sin(x_u-x_r) + \sin(x_r)\sin(x_u)\sin(x_v-x_s) \Big), \\
    \beta &= \sin(x_v)\sin(x_s-x_r)\cos(x_u) - \sin(x_s)\sin(x_v-x_u)\cos(x_r), \\
    \gamma &= \sin(x_s)\sin(x_v)\sin(x_u-x_r) - \sin(x_r)\sin(x_u)\sin(x_v-x_s).
  \end{align*}
\end{lemma}

\begin{proof}
  Let $P_1 = \sin(x_s)\sin(x_v)\sin(x_u-x_r)$ and $P_2 = \sin(x_r)\sin(x_u)\sin(x_v-x_s)$.
  Recall from the proof of Theorem \ref{NRSNUVTRIG2} that $\alpha = \frac{1}{2}(P_1+P_2)$
  and $\gamma = P_1-P_2$. Thus, $\alpha=0$ and $\gamma=0$ is equivalent to $P_1=0$ and
  $P_2=0$. Below, we consider three cases for $P_1=0$.

  When $\sin(x_s) = 0$, we have $\beta=\pm\sin(x_v)\sin(x_r)\cos(x_u)=0$ and
  $P_2=\pm\sin(x_r)\sin(x_u)\sin(x_v)=0$, which implies $\sin(x_r)=0$ or $\sin(x_v)=0$.
  Thus, $Z_s=Z_r=1$ or $Z_s=Z_v=1$. In the former case, by Theorem \ref{THEXFN}, $N(r, s)$
  is a scalar multiple of the identity matrix and commutes with any matrix. In the latter
  case, it reduces to $[N(r, s), N(u, v)]=0$ by Proposition \ref{NRSNUVCM}.

  Similarly, when $\sin(x_v) = 0$, we have $\sin(x_u)=0$ or $\sin(x_s)=0$. The former case
  implies that $N(u, v)$ is a scalar multiple of the identity matrix, and the latter
  reduces to $[N(r, s), N(u, v)]=0$.
  
  Finally, when $\sin(x_u-x_r) = 0$, we have $x_u = x_r + n\pi$ ($n \in \mathbb{Z}$),
  which yields
  \begin{align*}
    \beta &= (-1)^n \cos(x_r) \Big( \sin(x_v)\sin(x_s-x_r) - \sin(x_s)\sin(x_v-x_r) \Big) \\
    &= (-1)^{n+1} \cos(x_r)\sin(x_r)\sin(x_v-x_s) = 0.
  \end{align*}
  Combined with $P_2=0$, we obtain $\sin(x_r)=0$ or $\sin(x_v-x_s)=0$.
  The former reduces to $[N(r, s), N(u, v)]=0$, and the latter reduces to $N(r, s)=N(u, v)$.
  In all cases, $N(r, s)$ and $N(u, v)$ commute.
\end{proof}

\begin{proposition}\label{PROPNONSIMTRI}
  Assume that $k$ is a rational number such that $k \not\equiv 0, 4, 5, 6, 10, 11\pmod{12}$,
  and that the two matrices $N(r,s)$ and $N(u,v)$ do not commute.
  Then, $N(r,s)$ and $N(u,v)$ are not simultaneously triangularizable.
\end{proposition}

\begin{proof}
  We proceed by contradiction. Suppose that $N(r,s)$ and $N(u,v)$ do not commute but are
  simultaneously triangularizable.

  We have $\theta = \frac{k+1}{6}\pi$, and $\sin^2\theta = 1/4$ occurs only when
  $k \equiv 0, 4, 6, 10 \pmod{12}$. By assumption, $k \not\equiv 0, 4, 6, 10 \pmod{12}$,
  so $\sin^2\theta \neq 1/4$.

  Since the matrices $N(r,s)$ and $N(u,v)$ are simultaneously triangularizable but do not
  commute, by Theorem \ref{NRSNUVTS}, Proposition \ref{NRSNUVTRIG}, and Theorem
  \ref{NRSNUVTRIG2}, this condition is equivalent to the perfect square equation
  \begin{equation*}
    \alpha^2 + \beta^2 \sin^2\theta = \gamma^2 \left( \frac{1}{4} - \sin^2\theta \right).
  \end{equation*}

  Since $k$ is a rational number and $\sin^2\theta \neq 1/4$, Theorem
  \ref{NTSA2PB2S2EQG214S2} dictates that the only real solutions to this equation are the
  trivial ones, $\alpha = \beta = \gamma = 0$.
  
  However, by Lemma \ref{AEBEGE0CNRSNUV}, if $\alpha = \beta = \gamma = 0$, the matrices
  commute, which contradicts our initial assumption that $N(r,s)$ and $N(u,v)$ do not
  commute.
\end{proof}

Using the results obtained thus far, the main theorem \ref{MAINTHEOREM} can be easily
proved.

\begin{proof}[Proof of Theorem \ref{MAINTHEOREM}]
  By Remark \ref{REMMAINTHEOREM}, cases (1) and (2) of Theorem \ref{MAINTHEOREM}, as well
  as the case $k\equiv 5\pmod{6}$, can be excluded. Under these exclusions, Theorem
  \ref{LRSLUVIGNANDNCNRSNUV} guarantees the existence of distinct pairs of integers
  $(r, s)$ and $(u, v)$ such that $L(r, s), L(u, v)\in\Gamma(N)$ and the matrices $N(r, s)$
  and $N(u, v)$ do not commute.
  
  Consequently, by Proposition \ref{PROPNONSIMTRI}, for these pairs of integers $(r, s)$
  and $(u, v)$, the representation matrices $N(r, s)$ and $N(u, v)$ are not simultaneously
  triangularizable. This implies that the action of $L(r, s), L(u, v)\in\Gamma(N)$ on the
  basis $E_4(\tau)^{k/4}y_1(\tau)$ and $E_4(\tau)^{k/4}y_3(\tau)$ of the KZ equation is not
  reducible (i.e., it is irreducible). Therefore, in this case, the KZ equation admits no
  modular form solutions of weight $k$.
\end{proof}

Finally, we provide a corollary regarding the fact established by Theorem \ref{MAINTHEOREM}
and its relation to the Schwarzian differential equation corresponding to the KZ equation.

\begin{corollary}
  Let $n$ and $m$ be coprime positive integers, and suppose $m = 1$ or $m \ge 7$. Then,
  the two-dimensional representation $\rho|_{\Gamma(N)}$ of the principal congruence
  subgroup $\Gamma(N)$ associated with the solution $h(\tau)$ of the Schwarzian
  differential equation
  \begin{equation*}
    \{h, \tau\} = 2\pi^2\left(\frac{n}{m}\right)^2 E_4(\tau)
  \end{equation*}
  is irreducible.
\end{corollary}

\begin{proof}
  According to Saber and Sebbar \cite{SS2}, the solution $h$ of the Schwarzian
  differential equation $\{h, \tau\} = 2\pi^2(n/m)^2 E_4(\tau)$ is a $\rho$-equivariant
  function with respect to a suitable two-dimensional representation $\rho$ of the modular
  group (see \cite{SS1}). This equation is precisely the one satisfied by the ratio
  $h(\tau) = y_1(\tau)/y_3(\tau)$ of the independent solutions derived from the KZ
  equation of weight $k = \frac{6n}{m}-1$ by setting the parameter
  $\frac{k+1}{6} = \frac{n}{m}$.

  Classifying the condition on the weight $k$ by the denominator $m$ of the irreducible
  fraction yields the following:
  \begin{itemize}
    \item When $m=6$: Since $\gcd(n, 6)=1$, we have $n \equiv 1, 5 \pmod 6$, which implies
          $k \equiv 0, 4 \pmod 6$. This corresponds to case (1) of Theorem
          \ref{MAINTHEOREM}.
    \item When $m \in \{2, 3, 4, 5\}$: These correspond to case (2) of Theorem
          \ref{MAINTHEOREM}, respectively.
  \end{itemize}
  Therefore, the assumption ``$m = 1$ or $m \ge 7$'' is equivalent to not falling into
  cases (1) and (2) of Theorem \ref{MAINTHEOREM}, which means the original KZ equation
  admits no modular form solutions.

  When restricting the group to the principal congruence subgroup $\Gamma(N)$, if we
  assume that the representation $\rho\vert{}_{\Gamma(N)}$ associated with the equivariant
  function $h$ is reducible---meaning it has a one-dimensional invariant subspace (a
  common eigenvector)---then the original KZ equation must admit a modular form solution
  with respect to $\Gamma(N)$.
  
  For the case $m \ge 7$, as shown in the proof of Theorem \ref{MAINTHEOREM}, the 
  representation matrices of elements belonging to $\Gamma(N)$ are not simultaneously 
  triangularizable. In a two-dimensional linear representation, not being simultaneously 
  triangularizable means they do not share a one-dimensional invariant subspace. 
  Consequently, the representation $\rho\vert{}_{\Gamma(N)}$ is irreducible.
  On the other hand, for the case $m = 1$ (which corresponds to $k \equiv 5 \pmod 6$), 
  the matrix formulas are undefined since $\sin\theta = 0$. However, the non-existence 
  of a basis of modular form solutions is already completely established by the fact 
  that the solution space inherently involves quasi-modular forms (as shown by Kaneko 
  and Koike). This directly implies that the representation $\rho\vert{}_{\Gamma(N)}$ 
  is irreducible in this case as well, entirely independent of the connection matrix 
  argument.
\end{proof}

\section*{Acknowledgements}
The authors would like to express their deepest gratitude to Professor Abdellah Sebbar
and his assistant for their careful reading of the manuscript and for providing detailed
and constructive comments. In particular, we are immensely grateful for his crucial
suggestion of utilizing the norm argument, which fundamentally refined and completed
the proof of Theorem 4.5. Their contributions have significantly improved the clarity
and accuracy of this paper.

\bibliographystyle{plain}
\bibliography{kzssa}
\nocite{*}

\end{document}